\documentclass[11pt]{amsart}

\usepackage[T1]{fontenc}
\usepackage[utf8]{inputenc}
\usepackage{lmodern}
\usepackage{microtype}
\usepackage[margin=1.12in]{geometry}
\usepackage{amsmath,amssymb,amsthm,mathtools,comment,float}
\usepackage{booktabs}
\usepackage{enumitem}
\usepackage[colorlinks=true,citecolor=blue,linkcolor=blue,urlcolor=blue]{hyperref}
\usepackage[nameinlink,capitalize]{cleveref}

\numberwithin{equation}{section}

\newtheorem{theorem}{Theorem}[section]
\newtheorem{proposition}[theorem]{Proposition}
\newtheorem{lemma}[theorem]{Lemma}
\newtheorem{corollary}[theorem]{Corollary}
\theoremstyle{definition}
\newtheorem{definition}[theorem]{Definition}
\theoremstyle{remark}
\newtheorem{remark}[theorem]{Remark}

\newcommand{\R}{\mathbb R}
\newcommand{\Pj}{\mathbb P}
\newcommand{\C}{\mathbb C}
\newcommand{\cE}{\mathcal E}
\newcommand{\cH}{\mathcal H}

\newcommand{\ddbar}{\sqrt{-1}\partial \bar{\partial}}
\newcommand{\Ent}{\operatorname{Ent}}
\newcommand{\Ric}{\operatorname{Ric}}

\newcommand{\Aut}{\operatorname{Aut}}

\title{Positive scalar curvature K\"ahler metrics and shifted
\(K\)-stability}
\author{Zehao Sha}
\hypersetup{
  pdftitle={Positive scalar curvature Kahler metrics and shifted K-stability},
  pdfauthor={Zehao Sha}
}

\begin{document}

\begin{abstract}
In this paper, we show that the quantitative data underlying \(K\)-stability, originally developed in the study of canonical metrics, also detect the existence of positive scalar curvature K\"ahler metrics in a fixed class.
\end{abstract}

\maketitle
%\tableofcontents

\section{Introduction}
\label{sec:introduction}

\subsection{Constant scalar curvature K\"ahler metrics and \(K\)-stability}
\label{subsec:intro-canonical-metrics}

Let \((X,L)\) be a polarized manifold. The Yau--Tian--Donaldson (YTD) conjecture asked whether the existence of a constant scalar curvature K\"ahler (cscK) metric in \(2\pi c_1(L)\) is linked to an algebro-geometric stability condition, formulated in terms of test configurations and known as \(K\)-stability \cite{Yau93openproblem,Tian1997,Donaldson2002}. Since then, significant progress has been made concerning the necessary direction: Donaldson related the Donaldson--Futaki invariant to lower bounds for the Calabi functional, while Ross--Thomas introduced computable slope obstructions \cite{Donaldson2005,RossThomas2006}. Further, Stoppa \cite{Stoppa2009} proved that a polarized cscK manifold with discrete automorphism group is \(K\)-stable, and Berman--Darvas--Lu subsequently proved that a polarized manifold admitting a cscK metric is \(K\)-polystable \cite{BermanDarvasLu2020}. 

The converse direction is subtler. From a variational perspective, cscK metrics are precisely the smooth critical points of the Mabuchi \(K\)-energy \cite{Mabuchi1986}. The geometry of the space of K\"ahler potentials singles out geodesics as the natural paths along which to study the Mabuchi functional. In \cite{Chen2000Space}, Chen constructed weak geodesic segments, and Berman--Berndtsson established the convexity of the Mabuchi functional along such segments \cite{BermanBerndtsson2017}. In \cite{Darvas2015}, Darvas identified the \(d_1\)-metric completion of the space of smooth K\"ahler potentials with the finite-energy space \(\mathcal E^1\), on which the Mabuchi functional admits a canonical lower semicontinuous geodesically convex extension \cite{BermanDarvasLu2017}. This framework ultimately led to the resolution of the automorphism-relative \(d_1\)-metric formulation of Tian's properness conjecture \cite{Tian2000canonicalmetrics}: a K\"ahler class contains a cscK metric if and only if the Mabuchi functional is \(d_1\)-coercive modulo the identity component of the automorphism group. Darvas--Rubinstein formulated this criterion using the quotient \(d_1\)-metric and reduced its sufficiency direction to the regularity of finite-energy minimizers \cite{DarvasRubinstein2017}. Berman--Darvas--Lu proved the regularity of finite-energy minimizers on cscK manifolds \cite{BermanDarvasLu2020}, while Chen--Cheng established the a priori estimates and existence of solutions for the cscK equation, assuming the boundedness of the entropy \cite{ChenCheng2021I,ChenCheng2021II}. Building on Donaldson's geodesic stability conjecture and earlier constructions of weak geodesic rays associated with test configurations \cite{PhongSturm2007}, Darvas--Lu further developed the metric geometry of finite-energy geodesic rays and established an essentially optimal form of the geodesic stability criterion, detecting coercivity through the radial behavior of the Mabuchi functional along such rays \cite{DarvasLu2020}.

The radial viewpoint naturally leads to non-Archimedean geometry. A test configuration determines both an algebraic degeneration of \((X,L)\) and an asymptotic direction in the space of K\"ahler potentials. Boucksom--Hisamoto--Jonsson made this correspondence intrinsic by interpreting test configurations as non-Archimedean metrics on \(L\) and introducing non-Archimedean analogues of the Mabuchi and Aubin \(J\)-functionals \cite{BoucksomHisamotoJonsson2017}. They further showed that uniform \(K\)-stability is equivalent to the coercivity of the non-Archimedean Mabuchi functional with respect to the non-Archimedean Aubin \(J\)-functional on the space of positive non-Archimedean metrics. In \cite{BoucksomHisamotoJonsson2019}, Boucksom--Hisamoto--Jonsson subsequently established slope formulas identifying these non-Archimedean functionals with the asymptotic slopes of their Archimedean counterparts along compatible subgeodesic rays arising from test configurations. In this way, non-Archimedean geometry places algebraic degenerations and the radial asymptotics of the Mabuchi functional within a common framework.

Two recent works established Yau--Tian--Donaldson correspondences for the cscK problem in distinguished perspectives. Darvas--Zhang introduced the quantized \(K^\beta\)-energies and formulated uniform \(K^\beta\)-stability through their radial slopes along all K\"ahler test configurations, which yields a transcendental criterion beyond the polarized setting \cite{DarvasZhang2025}. By contrast, Boucksom--Jonsson worked on polarized manifolds and formulated a complete correspondence on the full space of finite-energy non-Archimedean metrics, characterizing cscK metrics, and more generally weighted extremal metrics with prescribed symmetries, by suitable forms of \(\widehat K\)-polystability \cite{BoucksomJonsson2025YTD}.

\subsection{Positive scalar curvature K\"ahler metrics}
\label{subsec:intro-positive-scalar-curvature}

The study of positive scalar curvature (PSC) metrics is a central theme in Riemannian geometry. The surgery theorem of Schoen--Yau \cite{SchoenYau1979Structure} and Gromov--Lawson \cite{GromovLawson1980} states that the existence of a PSC metric is preserved under surgeries of
codimension at least three. This mechanism has become one of the main tools in the construction of PSC metrics. For closed manifolds of dimension at least three, the existence problem
admits a variational formulation through the Yamabe invariant \cite{Kobayashi1987}. A manifold admits a PSC metric if and only if its Yamabe invariant is positive. Thus the existence problem is encoded by a diffeomorphism invariant defined by optimizing over conformal classes \cite{Trudinger1968,Aubin1976,Schoen1984}. On a closed spin manifold, the Lichnerowicz formula implies that a PSC metric forces the spin Dirac operator to have trivial kernel \cite{Lichnerowicz1963,Hitchin1974,GromovLawson1980Spin}.

In K\"ahler geometry, there is an additional layer of rigidity. The existence problem is no longer only a question of whether the underlying smooth manifold admits a PSC metric, but also depends on the complex structure and on the chosen K\"ahler class. On the other hand, positive scalar curvature also imposes a birational-geometric restriction on the underlying complex manifold. More precisely, the canonical bundle of a compact PSC K\"ahler manifold is not pseudoeffective. By Ou's characterization \cite{Ou2025}, such a manifold is uniruled.  This implication should be compared with Yang's theorem in Hermitian geometry: a compact complex manifold admits a Hermitian metric with positive Chern scalar curvature if and only if its canonical bundle is not pseudoeffective \cite{Yang2019}. 
%In higher dimension, uniruledness does not in general imply the existence of a positive scalar curvature K\"ahler metric, as shown by the counterexamples constructed in the companion paper \cite{Sha2026UniruledCounterexample}.

For compact K\"ahler surfaces, we have full classification of PSC manifolds. Combining the work of Yau \cite{Yau1974}, LeBrun \cite{LeBrun1995}, Friedman--Morgan \cite{FriedmanMorgan1997}, Brown \cite{Brown2026}, and the Enriques--Kodaira classification, the conditions: \((1)\) \(X\) admits a PSC K\"ahler metric; \((2)\) \(X\) admits a PSC Riemannian metric; \((3)\) the Kodaira dimension of \(X\) is \(-\infty\); are equivalent. Moreover, any PSC K\"ahler surface is obtained from \(\Pj^2\) or a ruled surface by a finite sequence of point blowups. This manifold-level classification does not, however, determine which individual K\"ahler classes contain PSC K\"ahler metrics.

\subsection{Main results}
\label{subsec:intro-main-results}

For the present paper, the decisive point is that \(K\)-stability contains more information than the binary distinction between stability and instability. In non-Archimedean geometry, the Donaldson--Futaki invariant admits an interpretation through the non-Archimedean Mabuchi functional \cite{BoucksomHisamotoJonsson2017}. Moreover, the asymptotic slope of the Mabuchi functional along a ray associated with a test configuration is identified with its non-Archimedean counterpart \cite{BoucksomHisamotoJonsson2019}. Thus uniform \(K\)-stability can be viewed as the strict positivity of the optimal lower bound of the ratio between the non-Archimedean Mabuchi functional and the non-Archimedean Aubin \(J\)-functional. Our starting point is to show that this quantitative stability scale admits another differential-geometric interpretation. For a K\"ahler class with positive average scalar curvature \(s\), after shifting the Mabuchi functional by \(sJ\), the positivity of the resulting threshold is equivalent to the existence of a positive scalar curvature K\"ahler metric. In particular, the numerical data underlying \(K\)-stability continue to carry geometric information even in the range where uniform \(K\)-stability fails.

Let \(X\) be a compact K\"ahler manifold and let \(\alpha\) be a K\"ahler class with positive average scalar curvature \(s>0\). Fix a reference K\"ahler metric \(\omega\in\alpha\), and denote by \(M_\omega\) the Mabuchi functional and by \(E_\omega\) the Aubin--Yau energy. Let \(\Omega >0\) be a smooth positive volume form satisfying
\[
        \int_X\Omega
        =
        \int_XS(\omega)\,\omega^n,
\]
where \(S(\omega)\) is the scalar curvature of \(\omega\). Consider the prescribed scalar curvature measure functional (\(F\)-functional) introduced in \cite{Sha2026PSCVariational}:
\[
        F_{\omega,\Omega}
        :=
        M_\omega
        -
        \overline S_\alpha E_\omega
        +
        L_{\omega,\Omega},
        \qquad
        L_{\omega,\Omega}(u)
        :=
        \frac{1}{\operatorname{Vol}_\omega(X)}\int_Xu\,\Omega .
\]
Its Euler--Lagrange equation is precisely
\begin{equation}\label{eq:scalar-measure}
            S(\omega_u)\,\omega_u^n=\Omega,\qquad \omega_u =\omega+\ddbar u .
\end{equation}
The essential feature of the \(F\)-functional is that the change of target measure \(\Omega\) does not affect the solvability of \eqref{eq:scalar-measure}. Indeed, by \cite[Theorem A]{Sha2026PSCVariational}, \(\alpha\) contains a PSC metric if and only if \eqref{eq:scalar-measure} is solvable for one, equivalently every, admissible measure \(\Omega\). We may therefore make the canonical choice relative to the reference metric \(\omega\), so that \(\Omega_\omega:=s\omega^n.\) Therefore the corresponding \(F\)-functional becomes
\[
        F_{\omega,\Omega}
        =
        M_\omega+s J_\omega,
\]
where \(J_\omega\) is the Aubin \(J\)-functional. The classical algebro-geometric formulation of \(K\)-stability is attached to a polarized manifold \((X,L)\), and hence, up to scaling, to a rational K\"ahler class. The existence of PSC K\"ahler metrics, however, is naturally posed for arbitrary real K\"ahler classes. We therefore work first in the transcendental framework in the sense of Darvas--Zhang~\cite{DarvasZhang2025}, with K\"ahler test configurations \cite{DervanRoss2017,SjostromDyrefelt2018}.

For \(\beta>0\), let \(M^\beta\) denote the quantized Mabuchi functional \cite{DarvasZhang2025}, and define
\begin{equation}
        F^\beta:=M^\beta+sJ.
        \label{eq:intro-F-beta}
\end{equation}
If \(\mathcal T\) is a K\"ahler test configuration for \((X,\alpha)\), we denote by \(F^\beta(\mathcal T)\) and \(J(\mathcal T)\) the radial slopes of the corresponding functionals along the associated geodesic ray.

Given \(\beta>0\), we say that \((X,\alpha)\) is \emph{shifted \(K^\beta\)-stable} if there exists \(\delta>0\) such that
\begin{equation}
        F^\beta(\mathcal T)
        \geq
        \delta J(\mathcal T)
        \label{eq:intro-shifted-Kbeta-stability}
\end{equation}
for every K\"ahler test configuration \(\mathcal T\) for \((X,\alpha)\). We retain the optimal lower bound in \eqref{eq:intro-shifted-Kbeta-stability} and then optimize over \(\beta>0\). More precisely, define the \emph{PSC threshold} by
\begin{equation}
        \sigma(X,\alpha)
        :=
        \sup_{\beta>0}
        \inf_{J(\mathcal T)>0}
        \frac{F^\beta(\mathcal T)}
             {J(\mathcal T)},
        \label{eq:intro-sigma-DZ}
\end{equation}
where the infimum is taken over all K\"ahler test configurations \(\mathcal T\) for \((X,\alpha)\) satisfying \(J(\mathcal T)>0\).

Our first result is the following transcendental criterion.
\begin{theorem}
\label{thm:main-DZ}
Let \(X\) be a compact K\"ahler manifold and let \(\alpha\) be a real K\"ahler class with positive average scalar curvature. Then the following conditions are equivalent:
\begin{enumerate}[label=\textup{(\roman*)}]
\item
\(\alpha\) contains a positive scalar curvature K\"ahler metric;

\item
\((X,\alpha)\) is shifted \(K^\beta\)-stable for some \(\beta>0\);

\item
\(\sigma(X,\alpha)>0\).
\end{enumerate}
\end{theorem}

We next pass to the polarized setting. Let \(L\) be an ample line bundle and set \(\alpha=2\pi c_1(L)\). Denote \(\cE^1_{\mathrm{na}}(L)\) by the space of finite-energy non-Archimedean potentials associated with \(L\), and let \(M_{\mathrm{na}}\) and \(J_{\mathrm{na}}\) denote the non-Archimedean Mabuchi and Aubin \(J\)-functionals, respectively. The \(F\)-functional has a non-Archimedean counterpart
\begin{equation}
        F_{\mathrm{na}}
        :=
        M_{\mathrm{na}}+sJ_{\mathrm{na}}
        \quad\text{on }\cE^1_{\mathrm{na}}.
        \label{eq:intro-F-na}
\end{equation}

We say that \((X,L)\) is \emph{shifted \(\widehat K\)-stable} if
\begin{equation}
        F_{\mathrm{na}}(\varphi)>0
        \qquad
        \text{for every }
        \varphi\in\cE^1_{\mathrm{na}}(L)
        \text{ with }J_{\mathrm{na}}(\varphi)>0.
        \label{eq:intro-shifted-Khat-stability}
\end{equation}
In particular, \(\alpha\) contains a PSC K\"ahler metric if and only if \( (X,L)\) is shifted \(\widehat K\)-stable (Theorem \ref{thm:BJ-PSC-criterion}). A key feature of this criterion is that the positivity in \eqref{eq:intro-shifted-Khat-stability} is automatically uniform: by Proposition \ref{prop:BJ-F-compatibility}, it is equivalent to the existence of \(\delta>0\) such that
\[
        F_{\mathrm{na}}(\varphi)
        \geq
        \delta J_{\mathrm{na}}(\varphi)
        \qquad
        \text{for every }
        \varphi\in\cE^1_{\mathrm{na}}(L).
\]
This leads to the \emph{non-Archimedean PSC threshold}
\begin{equation}
        \widehat{\sigma}(X,L)
        :=
        \inf_{\substack{
        \varphi\in\cE^1_{\mathrm{na}}\\
        J_{\mathrm{na}}(\varphi)>0}}
        \frac{F_{\mathrm{na}}(\varphi)}
             {J_{\mathrm{na}}(\varphi)}.
        \label{eq:intro-sigma-BJ}
\end{equation}
Moreover, \(\alpha\) contains a PSC K\"ahler metric if and only if \(\widehat{\sigma}(X,L)>0\).

Our second result identifies this non-Archimedean PSC threshold \(\widehat\sigma\) with the transcendental PSC threshold \(\sigma\) introduced above in the polarized setting.

\begin{theorem}
\label{thm:main-BJ}
Let \((X,L)\) be a polarized manifold, and assume that
\(2\pi c_1(L)\) has positive average scalar curvature. Then
\begin{equation}
        \sigma\bigl(X,2\pi c_1(L)\bigr)
        =
        \widehat{\sigma}(X,L).
        \label{eq:main-threshold-equality}
\end{equation}
\end{theorem}
Thus, in the polarized setting, the PSC threshold admits an intrinsic, \(\beta\)-independent non-Archimedean realization. In what follows, we write
\(\sigma:=\sigma\bigl(X,2\pi c_1(L)\bigr)=\widehat{\sigma}(X,L)\). The two distinguished values of this common threshold are \(0\) and \(s\). Combining Theorems \ref{thm:main-DZ} and \ref{thm:main-BJ} with Corollary \ref{cor:BJ-two-walls} gives the following metric interpretation.

\begin{corollary}
\label{cor:intro-metric-walls}
Let \((X,L)\) be as in Theorem \ref{thm:main-BJ}. Then:
\begin{enumerate}[label=\textup{(\roman*)}]
\item
\(2\pi c_1(L)\) contains a positive scalar curvature K\"ahler metric if and only if \(\sigma>0\);

\item
\(2\pi c_1(L)\) contains a unique cscK metric if and only if \(\sigma>s\).
\end{enumerate}
\end{corollary}

The critical value \(\sigma=s\) requires some attention. The equality \(\sigma=s\) implies \(M_{\mathrm{na}}\geq0\) on \(\cE^1_{\mathrm{na}}\),
but it does not determine the zero set of \(M_{\mathrm{na}}\). If \(\Aut^0(X,L)\) is trivial, then any cscK metric is unique, and Corollary \ref{cor:BJ-two-walls} would imply \(\sigma>s\). Hence no cscK
metric exists when \(\sigma=s\) in this case. On the other hand, suppose that a cscK metric exists and that \(\Aut^0(X,L)\) is nontrivial. By the Matsushima--Lichnerowicz theorem \cite{Matsushima1957,Lichnerowicz1958}, \(\Aut^0(X,L)\) is reductive and therefore admits a nonconstant real product direction \(\varphi\in\mathcal P_{\mathbb R}\). The \(\widehat K\)-polystability gives \(M_{\mathrm{na}}\geq0\), while \cite[Corollary~8.8]{BoucksomJonsson2025YTD} yields \(M_{\mathrm{na}}(\varphi)=0\). Since \(J_{\mathrm{na}}(\varphi)>0\), it follows that \(\sigma=s\). The converse does not follow from the numerical equality \(\sigma=s\). By \cite[Theorem~A]{BoucksomJonsson2025YTD}, cscK existence is equivalent to \(\widehat K\)-polystability, which in addition requires
\[
        M_{\mathrm{na}}(\varphi)=0
        \quad\Longrightarrow\quad
        \varphi\in\mathcal P_{\mathbb R}.
\]
Thus, at \(\sigma=s\), cscK existence depends on the zero set of \(M_{\mathrm{na}}\), information that is not captured by the threshold alone.

\subsection{Entropy regularization and \(K\)-stability}

The Boucksom--Jonsson entropy regularization conjecture was recently proved by Trusiani through special Fujita approximations \cite[Corollary~A]{Trusiani2026YTD} (and Appendix A in the latest version contributed by Boucksom). Applied with the trivial group action, the result gives that for every \(\varphi\in\cE^1_{\mathrm{na}}\) satisfying \(\Ent_{\mathrm{na}}(\varphi)<+\infty\), there exists a sequence \(\varphi_j\in\cH_{\mathrm{na}}\) such that
\[
        d_{1,\mathrm{na}}(\varphi_j,\varphi)\longrightarrow0,
        \qquad
        \Ent_{\mathrm{na}}(\varphi_j)
        \longrightarrow
        \Ent_{\mathrm{na}}(\varphi).
\]
Here \(\cH_{\mathrm{na}}\) is represented by normal ample test configurations for \((X,L)\).

Since \(M_{\mathrm{na}}-\Ent_{\mathrm{na}}\) and
\(J_{\mathrm{na}}\) are continuous in the strong topology, and points of \(\cE^1_{\mathrm{na}}\) at which
\(M_{\mathrm{na}}=+\infty\) do not affect the infimum, entropy regularization yields
\begin{equation} 
\begin{aligned} 
\inf_{\substack{ \psi\in\cH_{\mathrm{na}}\\ J_{\mathrm{na}}(\psi)>0}} \frac{M_{\mathrm{na}}(\psi)} {J_{\mathrm{na}}(\psi)} = \inf_{\substack{ \varphi\in\cE^1_{\mathrm{na}} \\J_{\mathrm{na}}(\varphi)>0}} \frac{M_{\mathrm{na}}(\varphi)} {J_{\mathrm{na}}(\varphi)} =\sigma-s. 
\end{aligned} 
\label{eq:intro-entropy-regularized-threshold} 
\end{equation} 
The left-hand side of \eqref{eq:intro-entropy-regularized-threshold} is directly related to classical \(K\)-stability. Indeed, applying \cite[Proposition~8.2]{BoucksomHisamotoJonsson2017} with trivial boundary, \(M_{\mathrm{na}}\ge 0\) on \(\cH_{\mathrm{na}}\) is equivalent to classical \(K\)-semistability, whereas \(M_{\mathrm{na}}\geq\delta J_{\mathrm{na}}\) on \(\cH_{\mathrm{na}}\) for some \(\delta>0\) is equivalent to uniform \(K\)-stability. The comparison with the Donaldson--Futaki invariant uses the normalized-base-change argument in the proof of that
proposition. Consequently, \((X,L)\) is \(K\)-semistable precisely when \(\sigma\geq s\), and uniformly \(K\)-stable precisely when \(\sigma>s\). In particular, \((X,L)\) is \(K\)-unstable if and only if \(\sigma<s\). At the critical value \(\sigma=s\), as explained above, it does not determine the existence of a cscK metric without further information on the zero set of \(M_{\mathrm{na}}\). The metric existence and stability conditions of the common threshold are summarized as follows.

\begin{table}[htbp]
\centering
\small
\renewcommand{\arraystretch}{1.3}
\begin{tabular}{@{}
p{0.16\textwidth}
p{0.38\textwidth}
p{0.35\textwidth}
@{}}
\toprule
Range of \(\sigma\)
&
Metrics in \(2\pi c_1(L)\)
&
 \(K\)-stability of \((X,L)\)
\\
\midrule
\(\sigma>s\)
&
A unique cscK metric
&
Uniformly \(K\)-stable
\\
\(\sigma=s\)
&
PSC K\"ahler metrics exist; cscK metric existence is not determined by \(\sigma\) alone
&
\(K\)-semistable
\\
\(0<\sigma<s\)
&
PSC K\"ahler metrics exist
&
\(K\)-unstable
\\
\(\sigma\leq0\)
&
No PSC K\"ahler metric
&
\(K\)-unstable
\\
\bottomrule
\end{tabular}
\caption{Metric existence and classical \(K\)-stability as determined by the common threshold \(\sigma\).}
\label{tab:intro-classical-threshold}
\end{table}

\subsection*{On the use of AI}

The author used ChatGPT 5.6 Sol for preliminary brainstorming, language editing, and consistency checks during the preparation of this manuscript. All mathematical statements, proofs, and citations were independently checked and remain the sole responsibility of the author.

\subsection*{Acknowledgement}

The author is sincerely grateful to Prof. Xiuxiong Chen for his continued encouragement and support. The author also thanks Longteng Chen, Song Sun, Jian Wang, Mingchen Xia, and Qi Yao for helpful discussions.

\section{Preliminaries}
\label{sec:preliminaries}

Let \(X\) be a compact K\"ahler manifold of complex dimension \(n\), and let \(\alpha\) be a K\"ahler class. Throughout this paper, we assume that \(\alpha\) has positive average scalar curvature. Fix a reference K\"ahler metric \(\omega\in\alpha\). The scalar curvature and the average scalar curvature of \(\omega\) are defined by
\[
        S(\omega)
        :=
        \frac{
         n\,\Ric(\omega)\wedge\omega^{n-1}
        }{
        \omega^n
        },\qquad s
        :=
        \frac{1}{\operatorname{Vol}_\omega(X)}
        \int_X S(\omega)\,\omega^n
        >0,
\]
where \(s\) depends only on the class \(\alpha\). For \(u\in C^\infty(X,\mathbb R)\), set
\[
        \omega_u
        :=
        \omega+\sqrt{-1}\,\partial\bar\partial u,
        \qquad
        \cH_\omega
        :=
        \left\{
        u\in C^\infty(X,\mathbb R):
        \omega_u>0
        \right\}.
\]
On \(\cH_\omega\), consider the \(L^1\)-type Finsler norm
\[
        \|\xi\|_{1,u}
        :=
        \frac{1}{\operatorname{Vol}_\omega(X)}
        \int_X|\xi|\,\omega_u^n.
\]
We denote the associated path-length metric by \(d_1\). The metric completion of \(\cH_\omega\) is naturally identified with the finite-energy space \(\cE^1(X,\omega)\); see \cite{Darvas2015}.

Recall that the Aubin--Yau energy on \(\cH_\omega\) is given by
\begin{equation}
        E(u)
        :=
        \frac{1}{
        (n+1)\operatorname{Vol}_\omega(X)}
        \sum_{j=0}^n
        \int_X
        u\,\omega_u^j\wedge\omega^{n-j}.
        \label{eq:normalized-AY-energy}
\end{equation}
The Aubin \(J\)-functional is
\begin{equation}
        J(u)
        :=
        \frac{1}{\operatorname{Vol}_\omega(X)}
        \int_Xu\,\omega^n-E(u).
        \label{eq:normalized-Aubin-J}
\end{equation}
Both functionals extend continuously to \(\cE^1(X,\omega)\). The Mabuchi functional \(M\) is characterized by \(M(0)=0\) and
\begin{equation}
        dM|_u(\dot \varphi)
        =
        -
        \frac{1}{\operatorname{Vol}_\omega(X)}
        \int_X
        \dot \varphi\,
        \bigl(S(\omega_u)-s\bigr)\omega_u^n.
        \label{eq:normalized-Mabuchi-variation}
\end{equation}
We use the same notation for its canonical lower semicontinuous extension \(M:\cE^1(X,\omega)\rightarrow\mathbb R\cup\{+\infty\}\). The reference metric determines the canonical smooth positive volume form \(s\,\omega^n\). The corresponding prescribed scalar-curvature measure functional is
\begin{equation}
        F(u)=M(u)+sJ(u).
        \label{eq:canonical-shifted-functional}
\end{equation}
Moreover, \(F\) also admits its canonical lower semicontinuous extension to \(\cE^1(X,\omega)\).

Since \(E(u+c)=E(u)+c\), the quotient \(\cE^1(X,\omega)/\mathbb R\) is naturally identified with the energy-normalized slice
\begin{equation}
        \cE^1_0(X,\omega)
        :=
        \left\{
        u\in\cE^1(X,\omega):
        E(u)=0
        \right\}.
        \label{eq:energy-normalized-slice}
\end{equation}
Since \(E\) is \(d_1\)-continuous and affine along finite-energy geodesics, this is a closed complete geodesic subspace of \((\cE^1(X,\omega),d_1)\). The \(F\)-functional is invariant under addition of constants and hence descend to \(\cE^1_0(X,\omega)\). We shall also use another normalization 
\[
        \cE^1_{\sup}(X,\omega)
        :=
        \left\{
        u\in\cE^1(X,\omega):
        \sup_Xu=0
        \right\}.
\]
In particular, there exists \(C\geq1\), depending only on \((X,\omega)\), such that
\begin{equation}
        C^{-1}J(u)-C
        \leq
        d_1(0,u)
        \leq
        CJ(u)+C,
        \qquad u\in\cE^1_0(X,\omega);
        \label{eq:energy-slice-d1-J}
\end{equation}
see \cite[equation~(62)]{Darvas2015}, whereas
\begin{equation}
        J(u)
        \leq
        d_1(0,u)
        \leq
        J(u)+C,
        \qquad \sup_X u =0;
        \label{eq:sup-slice-d1-J}
\end{equation}
see \cite[the top of p.23]{DarvasZhang2025}. Consequently, we have the following:
\begin{lemma}\label{lem:d1_coer_change}
    Let \(G:\cE^1(X,\omega) \to \mathbb R \cup \{+\infty\}\) be invariant under addition of constants. Then the following conditions are equivalent:
    \begin{enumerate}[label=\textup{(\roman*)}]
\item \(G\) is \(d_1\)-coercive on \(\cE^1_0(X,\omega)\);
\item \(G\) is \(J\)-coercive on \(\cE^1(X,\omega)/\R\);
\item \(G\) is \(d_1\)-coercive on \(\cE^1_{\sup}(X,\omega)\).
\end{enumerate}
The positive coercivity coefficients in these three formulations need
not be the same.
\end{lemma}

\section{The shifted \(K^\beta\)-stability}
\label{sec:darvas-zhang}

Let \(X\) be a compact K\"ahler manifold.  Fix a K\"ahler class \(\alpha\) with positive average scalar curvature \(s>0\), choose a reference metric \(\omega\in\alpha\). All functionals below are written in the normalization fixed in \(\S\)\ref{sec:preliminaries}.

\subsection{The \(d_1\)-coercivity of \(F^\beta\)}

Let \(E\) be the volume-normalized Aubin--Yau energy.  For \(\beta>0\), the \(\beta\)-entropy defined in \cite{DarvasZhang2025} is
\begin{equation}
        \Ent^\beta(u)
        :=
        \sup_{v\in\cE^1(X,\omega)}
        \left\{
        -\log\left(
        \frac{1}{\operatorname{Vol}_\omega(X)}
        \int_Xe^{\beta(v-u)}\omega^n
        \right)
        +\beta\bigl(E(v)-E(u)\bigr)
        \right\}.
        \label{eq:DZ-beta-entropy}
\end{equation}
The supremum is attained at the solution \(u^\beta\) of
\begin{equation}
        \omega_{u^\beta}^{\,n}
        =
        e^{\beta(u^\beta-u)}\omega^n,
        \label{eq:DZ-quantization-equation}
\end{equation}
and we have the simple but crucial formula as in \cite[Proposition 3.6]{DarvasZhang2025}:
\begin{equation}
        \Ent^\beta(u)
        =\beta\bigl(E(u^\beta)-E(u)\bigr).
        \label{eq:DZ-beta-energy-identity}
\end{equation}

For a smooth closed real \((1,1)\)-form \(\chi\), let \(\mathcal J_\chi\) be the contracted energy of \cite[Section~4]{DarvasZhang2025}, defined by
\[
\mathcal J_\chi:=n\,E^\chi_\omega-\overline{\chi}\,E_\omega,\qquad \overline{\chi}=\frac{n}{\operatorname{Vol}_\omega(X)} \int_X \chi \wedge \omega^{n-1}.
\]
Hence the \(\beta\)-entropy decomposition of \(M^\beta\) is given by
\begin{equation}
        M^\beta=\Ent^\beta-\mathcal J_{\Ric(\omega)}.
        \label{eq:DZ-M-decomposition}
\end{equation}
We then have the \(\beta\)-entropy decomposition formula for the quantized \(F^\beta\)-functional
\begin{equation}
        F^\beta:=\Ent^\beta-B_s, \quad \text{where} \quad B_s:=\mathcal J_{\Ric(\omega)}-sJ.
        \label{eq:DZ-quantized-F}
\end{equation}

The lower-order term \(B_s\) is uniformly bounded on \(d_1\)-bounded subsets of \(\cE^1_0(X,\omega)\) as shown in \cite[Lemma~2.19]{Sha2026PSCVariational}. For the quantization argument, one needs the following finer estimate, which controls the variation of \(B_s\) under \(u\mapsto u^\beta\) with an explicit decay rate.

\begin{lemma} \label{lem:DZ-shifted-lower-order}
Let \(a:=2^{-n}\).  There exists \(C>0\), depending only on \((X,\omega)\) and \(s\), such that, for every \(\beta>1\) and every \(u\in\cE^1(X,\omega)\) with \(\sup_Xu=0\), we have
\begin{equation}
\begin{split}
        \bigl|B_s(u^\beta)-B_s(u)\bigr|
        \leq C\beta^{-a}
        \bigl(
        \Ent^\beta(u)+d_1(0,u)+1+\log\beta
        \bigr).
\end{split}
\label{eq:DZ-shifted-lower-order-estimate}
\end{equation}
Moreover,
\begin{equation}
        |B_s(u)|\leq C\,d_1(0,u).
        \label{eq:DZ-lower-order-linear-growth}
\end{equation}
\end{lemma}

\begin{proof}
By \cite[Proposition~4.2]{DarvasZhang2025},
\begin{equation}
\begin{split}
        \left|
        \mathcal J_{\Ric(\omega)}(u^\beta)
        -\mathcal J_{\Ric(\omega)}(u)
        \right|
        \leq C\beta^{-a}
        \bigl(
        \Ent^\beta(u)+d_1(0,u)+1+\log\beta
        \bigr).
\end{split}
\label{eq:DZ-Ricci-J-estimate}
\end{equation}
It remains to estimate the Aubin \(J\)-term.  Since
\[
J(u)=\frac{1}{\operatorname{Vol}_\omega(X)}\int_Xu\,\omega^n-E(u),
\]
identity \eqref{eq:DZ-beta-energy-identity} gives
\begin{equation}
\begin{split}
        J(u^\beta)-J(u)
        ={}&
        \frac{1}{\operatorname{Vol}_\omega(X)}
        \int_X(u^\beta-u)\,\omega^n
        -\frac1\beta\Ent^\beta(u).
\end{split}
\label{eq:DZ-J-quantization-difference}
\end{equation}
For the first term on the right-hand side of \eqref{eq:DZ-J-quantization-difference}, \cite[Lemma~4.3, equation~(33)]{DarvasZhang2025} gives directly
\begin{equation}
        \left|
        \frac{1}{\operatorname{Vol}_\omega(X)}
        \int_X(u^\beta-u)\,\omega^n
        \right|
        \leq
        C\beta^{-a}
        \bigl(
        \Ent^\beta(u)+d_1(0,u)+1+\log\beta
        \bigr),
        \qquad a=2^{-n}.
\label{eq:DZ-MA0-quantization-estimate}
\end{equation}
Since \(\beta^{-1}\leq\beta^{-a}\), identity \eqref{eq:DZ-J-quantization-difference} therefore yields
\[
        \left|J(u^\beta)-J(u)\right|
        \leq
        C\beta^{-a}
        \bigl(
        \Ent^\beta(u)+d_1(0,u)+1+\log\beta
        \bigr).
\]
Together with \eqref{eq:DZ-Ricci-J-estimate}, this proves \eqref{eq:DZ-shifted-lower-order-estimate}.

Finally, the estimate used in the proof of \cite[Theorem~4.4]{DarvasZhang2025}, originating from \cite[Proposition~2.5]{DH2017}, gives
\[
\lvert\mathcal J_{\Ric(\omega)}(u)\rvert
\leq C\,d_1(0,u).
\]
Also, the Aubin \(J\)-functional is nonnegative and has linear \(d_1\)-growth. This proves
\eqref{eq:DZ-lower-order-linear-growth}.
\end{proof}

\begin{proposition}
\label{prop:DZ-shifted-quantization}
The following are equivalent on \(\cE^1_0(X,\omega)\):
\begin{enumerate}[label=\textup{(\roman*)}]
\item \(F\) is \(d_1\)-coercive;
\item \(F^\beta\) is \(d_1\)-coercive for some \(\beta>0\);
\item \(F^\beta\) is \(d_1\)-coercive for every sufficiently large
      \(\beta\).
\end{enumerate}
Moreover, for every \(t\in\mathbb R\), the same equivalence holds after replacing \(F\) and \(F^\beta\) by \(F-tJ\) and \(F^\beta-tJ\), respectively.
\end{proposition}

\begin{proof}
The implication \textup{(iii)}\(\Rightarrow\)\textup{(ii)} is immediate.  Since \(\Ent^\beta\leq\Ent\), one has \(F^\beta\leq F\). Therefore, \textup{(ii)} implies \textup{(i)}.  It remains to prove \textup{(i)}\(\Rightarrow\)\textup{(iii)}.

Assume that \(F\) is \(d_1\)-coercive on \(\cE^1_0(X,\omega)\).  By Lemma~\ref{lem:d1_coer_change}, there exist \(\delta,C_0>0\) such that
\begin{equation}
        F(u)
        \geq
        \delta d_1(0,u)-C_0,
        \qquad
        u\in\cE^1_{\sup}(X,\omega).
        \label{eq:DZ-F-coercive-sup-slice}
\end{equation}
Fix \(u\in\cE^1_{\sup}(X,\omega)\). By \cite[Lemma~3.8]{DarvasZhang2025} and Lemma~\ref{lem:DZ-shifted-lower-order}, we have
\begin{align*}
        F^\beta(u)
        &=\Ent^\beta(u)-B_s(u)\\
        &\geq
        F\left(u^\beta-\sup_Xu^\beta\right)
        +B_s(u^\beta)-B_s(u)\\
        &\geq
        \delta d_1\left(0,u^\beta-\sup_Xu^\beta\right)-C
        -C\beta^{-a}
        \bigl(
        \Ent^\beta(u)+d_1(0,u)+1+\log\beta
        \bigr).
\end{align*}
It follows from the triangle inequality, together with \cite[Proposition~4.1 and Lemma~4.3]{DarvasZhang2025} that
\[
        d_1\left(0,u^\beta-\sup_Xu^\beta\right)
        \geq
        d_1(0,u)-C
        -C\beta^{-a}
        \bigl(
        \Ent^\beta(u)+d_1(0,u)+1+\log\beta
        \bigr).
\]
Consequently, after changing \(C\),
\begin{equation}
        F^\beta(u)
        \geq
        \delta d_1(0,u)-C
        -
        \eta_\beta
        \bigl(
        \Ent^\beta(u)+d_1(0,u)+1+\log\beta
        \bigr),
\label{eq:DZ-before-absorption}
\end{equation}
where \(\eta_\beta=C\beta^{-a}\to0\).

Since \(\Ent^\beta=F^\beta+B_s\), \eqref{eq:DZ-lower-order-linear-growth} in Lemma \ref{lem:DZ-shifted-lower-order} allows us to absorb the entropy term in \eqref{eq:DZ-before-absorption}.  Hence, we obtain
\[
        (1+\eta_\beta)F^\beta(u)
        \geq
        (\delta-C\eta_\beta)d_1(0,u)
        -C-C\eta_\beta\log\beta.
\]
For all sufficiently large \(\beta\), one has \(\delta-C\eta_\beta>0\), while
\(\eta_\beta\log\beta\to0\). Thus \(F^\beta\) is \(d_1\)-coercive on \(\cE^1_{\sup}(X,\omega)\) for every sufficiently large \(\beta\).  By Lemma~\ref{lem:d1_coer_change},
\(F^\beta\) is therefore \(d_1\)-coercive on \(\cE^1_0(X,\omega)\).

The final assertion follows by replacing \(s\) throughout by \(s-t\).  Indeed,
\[
        F-tJ=M+(s-t)J,
        \qquad
        F^\beta-tJ=M^\beta+(s-t)J,
\]
and the preceding estimates remain valid for every fixed \(t\in\mathbb R\), with constants allowed to depend on \(t\).
\end{proof}

\subsection{K\"ahler test configurations and PSC K\"ahler metrics}

We first recall the transcendental notion of a K\"ahler test configuration, which was originally introduced in \cite[Definition~2.10]{DervanRoss2017} and \cite[Definitions~1.3, 3.2, and~3.5]{SjostromDyrefelt2018}. We use the precise formulation adopted in \cite[Definition~5.4]{DarvasZhang2025}.

\begin{definition} \label{def:DZ-Kahler-test-configuration}
A \emph{K\"ahler test configuration} for \((X,\alpha)\) is a pair \(\mathcal T=(\mathcal X,\mathcal A)\) with the following properties:
\begin{enumerate}[label=\textup{(\roman*)}]
\item \(\mathcal X\) is a normal compact complex space equipped with a flat surjective morphism \(\pi:\mathcal X\longrightarrow\Pj^1\) and a holomorphic \(\C^*\)-action lifting the standard action on \(\Pj^1\);

\item there is a \(\C^*\)-equivariant biholomorphism \(\mathcal X\setminus\mathcal X_0\simeq
X\times\bigl(\Pj^1\setminus\{0\}\bigr)\) over \(\Pj^1\setminus\{0\}\), where \(\mathcal X_0:=\pi^{-1}(0)\);

\item \(\mathcal A\in  H_{\mathrm{BC}}^{1,1}(\mathcal X,\R)\) is a \(\C^*\)-invariant Bott--Chern class whose restriction under the above trivialization is \(\mathcal A\big|_{\mathcal X\setminus\mathcal X_0}=p_1^*\alpha\), where \(p_1:X\times\Pj^1\to X\) is the first projection;

\item \(\mathcal A\) is relatively K\"ahler, namely \(\mathcal A+
              c\,\pi^*c_1\bigl(\mathcal O_{\Pj^1}(1)\bigr)\) is a K\"ahler class on \(\mathcal X\) for some \(c>0\).
\end{enumerate}
\end{definition}
Fixing \(\omega\in\alpha\), a K\"ahler test configuration determines an associated weak geodesic ray \(\ell^{\mathcal T}\); see \cite[Lemma~4.6]{SjostromDyrefelt2018} and \cite{CTW2017}. After adding a linear function of \(t\), we normalize the ray by \(\sup_Xu_t^{\mathcal T}=0\) with \(t\geq0\). For two finite-energy geodesic rays \(\ell=\{u_t\}\) and \(\widetilde\ell=\{\widetilde u_t\}\) starting from \(0\), their chordal distance is
\[
        d_1^c(\ell,\widetilde\ell)
        :=
        \lim_{t\to\infty}
        \frac{d_1(u_t,\widetilde u_t)}{t}.
\]
We keep the standard notation \(M\{\ell\}\), \(J\{\ell\}\), and \(F\{\ell\}\) for the radial slopes of the Mabuchi, Aubin \(J\), and \(F\)-functionals. For a finite-energy geodesic ray \(\ell=\{u_t\}_{t\geq0}\), set
\begin{equation}
\begin{split}
        M^\beta\{\ell\}
        &:=
        \liminf_{t\to\infty}
        \frac{M^\beta(u_t)}{t},\\
        F^\beta\{\ell\}
        &:=
        \liminf_{t\to\infty}
        \frac{F^\beta(u_t)}{t}
        =
        M^\beta\{\ell\}+sJ\{\ell\}.
\end{split}
\label{eq:DZ-quantized-radial-conventions}
\end{equation}
The second identity follows from \(F^\beta=M^\beta+sJ\) and the existence of the radial \(J\)-slope. Recall \eqref{eq:sup-slice-d1-J}, for a ray normalized by \(\sup_Xu_t=0\), one has
\[
        J(u_t)
        \leq d_1(0,u_t)
        \leq J(u_t)+C;
\]
Since \(\ell\) has constant \(d_1\)-speed, denoted by \(v_1(\ell)\), division by \(t\) gives the exact radial identity
\begin{equation}
        J\{\ell\}=v_1(\ell).
        \label{eq:DZ-radial-J-speed}
\end{equation}
The radial slopes are independent of the reference metric. Indeed, under the natural change of reference metrics, \(M^\beta-M=\Ent^\beta-\Ent\), and replacing one smooth reference volume form by another changes the right-hand side by a uniformly bounded quantity, which does not affect its radial slope. Moreover, the radial \(J\)-slope is likewise invariant by \eqref{eq:DZ-radial-J-speed}. 

For a K\"ahler test configuration \(\mathcal T\), define
\[
        F^\beta(\mathcal T)
        :=
        F^\beta\{\ell^{\mathcal T}\},
        \qquad
        J(\mathcal T)
        :=
        J\{\ell^{\mathcal T}\}.
\]
Compatible rays obtained from different auxiliary choices are parallel.  The \(d_1^c\)-continuity of the radial \(M^\beta\)- and \(J\)-functionals therefore shows that these quantities depend only on \(\mathcal T\); see
\cite[Theorem~5.2 and the proof of Theorem~1.1]
{DarvasZhang2025}.

\begin{definition}[Shifted \(K^\beta\)-stability]
\label{def:DZ-shifted-Kbeta-stability}
Let \(\beta>0\). We say that \((X,\alpha)\) is \emph{shifted \(K^\beta\)-stable} if there exists \(\delta>0\) such that \(F^\beta(\mathcal T) \geq\delta J(\mathcal T)\) for every K\"ahler test configuration \(\mathcal T\) for \((X,\alpha)\).
\end{definition}

\begin{theorem}
\label{thm:DZ-PSC-criterion}
Let \(X\) be a compact K\"ahler manifold and let \(\alpha\) be a K\"ahler class with average scalar curvature \(s>0\).  Then the following are equivalent:
\begin{enumerate}[label=\textup{(\roman*)}]
\item \(\alpha\) contains a positive scalar curvature K\"ahler metric;

\item \((X,\alpha)\) is shifted \(K^\beta\)-stable for some \(\beta>0\).
\end{enumerate}
\end{theorem}

\begin{proof}
Suppose first that \(\alpha\) contains a positive scalar curvature K\"ahler metric.  By \cite[Theorem A]{Sha2026PSCVariational}, \(F\) is \(d_1\)-coercive on \(\cE^1_0(X,\omega)\). Proposition \ref{prop:DZ-shifted-quantization} then implies that \(F^\beta\) is \(d_1\)-coercive for every sufficiently large \(\beta\) on \(\cE^1_0(X,\omega)\). Fix one such \(\beta>0\). By Lemma~\ref{lem:d1_coer_change}, there exist \(\varepsilon>0\) and \(C>0\) such that
\begin{equation}
        F^\beta(u)
        \geq
        \varepsilon d_1(0,u)-C,
        \qquad
        u\in\cE^1_{\sup}(X,\omega).
        \label{eq:DZ-Fbeta-coercive-ray-bound}
\end{equation}

Let \(\mathcal T\) be a K\"ahler test configuration and normalize its compatible ray \(\ell^{\mathcal T}=\{u_t^{\mathcal T}\}\) by \(\sup_Xu_t^{\mathcal T}=0\).  Dividing
\eqref{eq:DZ-Fbeta-coercive-ray-bound} by \(t\) and letting \(t\to\infty\), we obtain
\[
\begin{split}
        F^\beta(\mathcal T)
        \geq
        \varepsilon
        \lim_{t\to\infty}
        \frac{d_1(0,u_t^{\mathcal T})}{t}  
        =
        \varepsilon v_1(\ell^{\mathcal T})
        =
        \varepsilon J(\mathcal T),
\end{split}
\]
where the last equality is \eqref{eq:DZ-radial-J-speed}.  Hence \((X,\alpha)\) is shifted
\(K^\beta\)-stable.

Conversely, suppose that \((X,\alpha)\) is shifted \(K^\beta\)-stable for some fixed \(\beta>0\).  Let \(\ell=\{u_t\}_{t\geq0}\) be a nonconstant finite-energy geodesic
ray, normalized by \(\sup_Xu_t=0\). If \(F\{\ell\}=+\infty\), then there is nothing to prove.  Suppose that \(F\{\ell\}<+\infty\).  Since \(J\{\ell\}=v_1(\ell)<+\infty\), it follows from
\(F=M+sJ\) that
\[
        M\{\ell\}
        =
        F\{\ell\}-sJ\{\ell\}
        <+\infty.
\]
By the statement used in pvoving \cite[Theorem~5.5]{DarvasZhang2025}, obtained there by combining \cite[Proposition~6.3.1 and Theorem~5.1.7]
{MesquitaPiccione2024}, every finite-energy geodesic ray with bounded Mabuchi slope is \(d_1^c\)-approximable by rays of K\"ahler test configurations. Hence there exist K\"ahler test configurations \(\mathcal T_j\), with compatible rays \(\ell_j:=\ell^{\mathcal T_j}\), such that
\[
        d_1^c\bigl(
        \ell^{\mathcal T_j},\ell
        \bigr)
        \longrightarrow0.
\]
For every \(j\), shifted \(K^\beta\)-stability gives
\[
        F^\beta\{\ell^{\mathcal T_j}\}
        \geq
        \delta J\{\ell^{\mathcal T_j}\}.
\]
By \cite[Theorem~5.2]{DarvasZhang2025}, \(M^\beta\{\cdot\}\) is \(d_1^c\)-continuous. Moreover, \(J\{\cdot\}\) is also \(d_1^c\)-continuous, as follows from \cite[Lemma~4.15]{Darvas2015} and \cite[Lemma~A.2]{BermanBoucksomJonsson2021}. Consequently, \(F^\beta\{\cdot\}\) is \(d_1^c\)-continuous.  We may therefore pass to the limit and obtain
\begin{equation}
        F^\beta\{\ell\}
        \geq
        \delta J\{\ell\}.
        \label{eq:DZ-Fbeta-all-rays}
\end{equation}

Since \(F^\beta\leq F\), taking radial slopes and using
\eqref{eq:DZ-Fbeta-all-rays}, we obtain
\[
        F\{\ell\}
        \geq
        \delta J\{\ell\}
\]
for every nonconstant finite-energy geodesic ray normalized by \(\sup \ell_t =0\). Let
\(\widetilde\ell=\{\widetilde u_t\}_{t\geq0}\subset\cE^1_0(X,\omega)\) be a constant-speed geodesic ray starting from \(0\). By \cite[Theorem~3.4]{Darvas2017WeakGeodesics},
\(t\mapsto\sup_X\widetilde u_t\) is linear. Hence \(u_t:=\widetilde u_t-\sup_X\widetilde u_t\) is again a geodesic ray and satisfies \(\sup_Xu_t=0\). Since \(F\) and \(J\) are invariant under addition of constants,
\[
        F\{\widetilde\ell\}
        =
        F\{\ell\},
        \qquad
        J\{\widetilde\ell\}
        =
        J\{\ell\}.
\]
Combining with \eqref{eq:energy-slice-d1-J}, we therefore obtain
\[
        F\{\widetilde\ell\}
        \geq
        \delta C^{-1}v_1(\widetilde\ell).
\]
Thus \(F\) is uniformly geodesically stable. By \cite[Theorem~A]{Sha2026PSCVariational}, \(\alpha\) contains a positive scalar curvature K\"ahler metric.
\end{proof}

The PSC threshold introduced in \(\S\)\ref{subsec:intro-main-results} is 
\begin{equation}
        \sigma(X,\alpha)
        :=
        \sup_{\beta>0}
        \inf_{J(\mathcal T)>0}
        \frac{F^\beta(\mathcal T)}
             {J(\mathcal T)}.
        \label{eq:DZ-threshold}
\end{equation}
The two distinguished numerical walls for \(\sigma\) are located at \(0\) and \(s\): the first detects positive scalar curvature, while the second is obtained by removing the fixed shift \(sJ\) from \(F^\beta=M^\beta+sJ\) and recovering the quantized Mabuchi functional.

\begin{corollary}
\label{cor:DZ-two-walls}
Let \(X\) be a compact K\"ahler manifold, and let \(\alpha\) be a K\"ahler class with positive average scalar curvature \(s>0\). Then:
\begin{enumerate}[label=\textup{(\roman*)}]
\item \(\alpha\) contains a PSC K\"ahler metric if and only if \(\sigma(X,\alpha)>0\);

\item \(\alpha\) contains a unique cscK metric if and only if \(\sigma(X,\alpha)>s\).
\end{enumerate}
\end{corollary}

\begin{proof}
By the definition of \(\sigma(X,\alpha)\), the inequality \(\sigma(X,\alpha)>0\) is equivalent to shifted \(K^\beta\)-stability for some \(\beta>0\).  The first equivalence therefore follows from Theorem \ref{thm:DZ-PSC-criterion}.

For the second equivalence, \( \sigma(X,\alpha)>s\) if and only if there exist \(\beta>0\) and \(\delta>0\) such that for every K\"ahler test configuration, we have
\begin{equation*}
F^\beta(\mathcal T)\geq(s+\delta)J(\mathcal T),
\end{equation*}
which is equivalent to
\[
        M^\beta(\mathcal T)
        \geq\delta J(\mathcal T)
\]
for some \(\beta>0\), some \(\delta>0\), and every K\"ahler test
configuration.  By \cite[Theorem~5.5]{DarvasZhang2025}, this is equivalent to the
existence of a unique cscK metric in the class.
\end{proof}

By combining Theorem \ref{thm:DZ-PSC-criterion} and Corollary \ref{cor:DZ-two-walls}, we obtain Theorem \ref{thm:main-DZ}.

\section{PSC K\"ahler metrics in non-Archimedean geometry}
\label{sec:boucksom-jonsson}

Let \((X,L)\) be a polarized manifold such that \(\alpha=2\pi c_1(L)\) has positive average scalar curvature \(s>0\). The \(F\)-functional admits an intrinsic non-Archimedean counterpart. The purpose of this section is to express the preceding positive scalar curvature criterion in the polarized setting.

\subsection{The non-Archimedean \(F\)-functional}

We follow the notation of \cite{BoucksomJonsson2025YTD}. Fix the ample line bundle \(L\). The space \(\cH_{\mathrm{na}}\) of Fubini--Study potentials is canonically identified with the isomorphism classes of normal relatively ample test configurations
for \((X,L)\).  Equivalently, in the terminology of \cite{BoucksomHisamotoJonsson2017}, it is the quotient of semiample test configurations by pullback equivalence, in the case that each equivalence class admits a unique normal ample representative. The \(d_{1,\mathrm{na}}\)-completion of \(\cH_{\mathrm{na}}\) is the finite-energy space
\(\cE^1_{\mathrm{na}}\). The non-Archimedean Mabuchi functional \(M_{\mathrm{na}}\colon\cH_{\mathrm{na}}\to\mathbb Q\) admits a natural lower semicontinuous extension \(M_{\mathrm{na}}\colon\cE^1_{\mathrm{na}} \rightarrow\mathbb R\cup\{+\infty\}\), denoted by the same symbol. Write \(J_{\mathrm{na}}\) as the non-Archimedean Aubin \(J\)-functional. We then define
\begin{equation}
        F_{\mathrm{na}}
        :=
        M_{\mathrm{na}}
        +
        sJ_{\mathrm{na}}
        \quad\text{on }\cE^1_{\mathrm{na}}.
        \label{eq:BJ-NA-F-definition}
\end{equation}

Let \(\cE^1_{\mathrm{rad}}\) denote the space of finite-energy radial directions, namely the space of equivalence classes of finite-energy geodesic rays in \(\cE^1(X,\omega)\) modulo parallelism. If \(\varphi=[\ell]\in\cE^1_{\mathrm{rad}}\), where \(\ell=\{u_t\}_{t\geq0}\), the radial slope \(F\{\ell\}\) depends only on the direction
\(\varphi\), and not on the representative \(\ell\); see \cite[Proposition~2.13]{BoucksomJonsson2025YTD}.

Every finite-energy geodesic ray determines a non-Archimedean potential. Conversely, for each \(\psi\in\cE^1_{\mathrm{na}}\) and each smooth reference potential,
there exists a unique maximal geodesic ray \(\ell_\psi\) having
non-Archimedean potential \(\psi\). Taking its direction defines an
\(\R_{>0}\)-equivariant isometric embedding
\begin{equation}
        \iota:
        (\cE^1_{\mathrm{na}},d_{1,\mathrm{na}})
        \hookrightarrow
        (\cE^1_{\mathrm{rad}},d_{1,\mathrm{rad}}),
        \qquad
        \psi\longmapsto[\ell_\psi].
        \label{eq:BJ-NA-radial-embedding}
\end{equation}
See \cite[\S5.1]{BoucksomJonsson2025YTD}. We henceforth identify \(\cE^1_{\mathrm{na}}\) with its closed image in \(\cE^1_{\mathrm{rad}}\). Let \(\varphi=[\ell]\in\cE^1_{\mathrm{rad}}\). By \cite[Theorem~C]{BoucksomJonsson2025YTD}, together with the identity \(J\{\ell\}=J_{\mathrm{na}}(\varphi)\) on \(\cE^1_{\mathrm{na}}\) (see \cite[Lemma~4.3, Theorem~6.6 and Corollary~6.7]{BermanBoucksomJonsson2021}), one has
\begin{equation}
        F\{\ell\}
        =
        \begin{cases}
        F_{\mathrm{na}}(\varphi),
        &\varphi\in\cE^1_{\mathrm{na}},\\
        +\infty,
        &\varphi\notin\cE^1_{\mathrm{na}}.
        \end{cases}
        \qquad\text{for }
       [\ell]= \varphi\in\cE^1_{\mathrm{rad}}.
        \label{eq:BJ-F-radial-identity}
\end{equation}

For \(\varphi\in\cE^1_{\mathrm{na}}\), after choosing the unique translate satisfying \(E_{\mathrm{na}}(\varphi)=0\), there exists \(C\geq1\), depending only on the dimension, such that
\begin{equation}
        C^{-1}J_{\mathrm{na}}(\varphi)
        \leq
        d_{1,\mathrm{na}}(0,\varphi)
        =
        d_{1,\mathrm{rad}}(0,\varphi)
        \leq
        C J_{\mathrm{na}}(\varphi).
        \label{eq:BJ-J-d1-comparison}
\end{equation}
The comparison between \(d_{1,\mathrm{na}}\) and \(J_{\mathrm{na}}\) follows from  \cite[Lemma~5.6 and Theorem~5.5]{BJ2025AIF}, while the equality with the radial metric follows from \cite[\S5.1]{BoucksomJonsson2025YTD}.

\begin{proposition}
\label{prop:BJ-F-compatibility}
Let \((X,L)\) be a polarized manifold such that \(\alpha=2\pi c_1(L)\) has positive average scalar curvature \(s>0\), and fix \(\omega\in\alpha\). Then the following are
equivalent:
\begin{enumerate}[label=\textup{(\roman*)}]
\item
\(F\) is \(d_1\)-coercive on \(\cE^1_0(X,\omega)\);

\item
\(F_{\mathrm{na}}(\varphi)>0\) for every \(\varphi\in\cE^1_{\mathrm{na}}\) such that \(J_{\mathrm{na}}(\varphi)>0\);

\item
there exists \(\delta>0\) such that \(F_{\mathrm{na}}(\varphi)\geq\delta J_{\mathrm{na}}(\varphi)\) for every \(\varphi\in\cE^1_{\mathrm{na}}\).
\end{enumerate}
\end{proposition}

\begin{proof}
We first verify that \(F\) is strongly lower semicontinuous in the sense of
\cite[Definition~2.20]{BoucksomJonsson2025YTD}. In particular, \cite[\S2.5]{Sha2026PSCVariational} shows that \(F\) is \(d_1\)-lower semicontinuous and geodesically convex on \(\cE^1_0(X,\omega)\). Moreover, \cite[Proposition~4.1]{Sha2026PSCVariational} shows that the intersection of every \(F\)-sublevel set with a closed \(d_1\)-ball is relatively compact. Such an intersection is closed by the lower semicontinuity of \(F\), and hence compact.

Assume first that \textup{(i)} holds. Let
\(\varphi\in\cE^1_{\mathrm{na}}\), normalized by
\(E_{\mathrm{na}}(\varphi)=0\), and let
\(\ell_\varphi=\{u_t\}_{t\geq0}\) be the maximal geodesic ray
directed by \(\varphi\) and emanating from \(0\). By
\cite[Corollary~6.7]{BermanBoucksomJonsson2021},
\[
        E(u_t)
        =
        E(0)+tE_{\mathrm{na}}(\varphi)
        =
        0.
\]
Thus \(\ell_\varphi\subset\cE^1_0(X,\omega)\), and coercivity gives
\[
        F(u_t)\geq\delta\, d_1(0,u_t)-C_0.
\]
Dividing by \(t\) and letting \(t\to\infty\), we obtain
\[
\begin{split}
        F_{\mathrm{na}}(\varphi)
        =
        F\{\ell_\varphi\}
        \geq
        \delta \, d_{1,\mathrm{rad}}(0,\varphi)
        =
        \delta \, d_{1,\mathrm{na}}(0,\varphi)
        \geq
        \frac{\delta}{C}
        J_{\mathrm{na}}(\varphi),
\end{split}
\]
where we used \eqref{eq:BJ-F-radial-identity} and \eqref{eq:BJ-J-d1-comparison}. Hence \textup{(iii)} holds. The implication \(\textup{(iii)}\Rightarrow\textup{(ii)}\) is immediate.

Finally, assume \textup{(ii)} and suppose that \textup{(i)} fails. Applying \cite[Corollary~2.28]{BoucksomJonsson2025YTD} to \(\cE^1_0(X,\omega)\), with the trivial group action, gives a unit-speed geodesic ray \(\ell=\{u_t\}_{t\geq0}\subset\cE^1_0(X,\omega)\) starting from \(0\) along which \(F\) is nonincreasing. Then \(F\{\ell\}\leq0\). Let \(\varphi=[\ell]\in\cE^1_{\mathrm{rad}}\) be its direction. By \eqref{eq:BJ-F-radial-identity}, the finiteness of \(F\{\ell\}\) forces \(\varphi\in\cE^1_{\mathrm{na}}\), and
\[
        F_{\mathrm{na}}(\varphi)
        =
        F\{\ell\}
        \leq0.
\]
Moreover, \(E_{\mathrm{na}}(\varphi)=0\) and \(d_{1,\mathrm{na}}(0,\varphi)=d_{1,\mathrm{rad}}(0,\varphi)=1\). It follows from \eqref{eq:BJ-J-d1-comparison} that \(J_{\mathrm{na}}(\varphi)\geq C^{-1}>0\), contradicting \textup{(ii)}. Therefore \textup{(i)} holds.
\end{proof}

We can introduce the concept of shifted \(\widehat K\)-stability on a polarized manifold \((X,L)\).
\begin{definition}[Shifted \(\widehat K\)-stability]
\label{def:BJ-shifted-Khat-stability}
Let \((X,L)\) be a polarized manifold such that \(\alpha=2\pi c_1(L)\) has positive average scalar curvature.
\begin{enumerate}[label=\textup{(\roman*)}]
\item We say that \((X,L)\) is \emph{shifted \(\widehat K\)-stable} if \(F_{\mathrm{na}}(\varphi)>0\) for every \(\varphi\in\cE^1_{\mathrm{na}}\) such that \(J_{\mathrm{na}}(\varphi)>0\).

\item We say that \((X,L)\) is \emph{uniformly shifted \(\widehat K\)-stable} if there exists \(\delta>0\) such that \(F_{\mathrm{na}}(\varphi)\geq \delta\, J_{\mathrm{na}}(\varphi)\) for every \(\varphi\in\cE^1_{\mathrm{na}}\).
\end{enumerate}
\end{definition}

\begin{remark}
By Proposition \ref{prop:BJ-F-compatibility}, shifted \(\widehat K\)-stability and uniform shifted \(\widehat K\)-stability are equivalent.
\end{remark}

\begin{theorem}
\label{thm:BJ-PSC-criterion}
Let \((X,L)\) be a polarized manifold such that \(\alpha=2\pi c_1(L)\) has positive average scalar curvature. Then the following are equivalent:
\begin{enumerate}[label=\textup{(\roman*)}]
\item
\(\alpha\) contains a PSC K\"ahler metric;

\item
\((X,L)\) is shifted \(\widehat K\)-stable.
\end{enumerate}
\end{theorem}

\begin{proof}
By \cite[Theorem A]{Sha2026PSCVariational}, the existence of a positive scalar curvature K\"ahler metric in \(\alpha\) is equivalent to the \(d_1\)-coercivity of \(F\) on \(\cE^1_0(X,\omega)\).  The result now follows from Proposition \ref{prop:BJ-F-compatibility}.
\end{proof}

Recall the non-Archimedean PSC threshold defined by
\begin{equation}
        \widehat{\sigma}(X,L)
        :=
        \inf_{\substack{
        \varphi\in\cE^1_{\mathrm{na}}\\
        J_{\mathrm{na}}(\varphi)>0}}
        \frac{F_{\mathrm{na}}(\varphi)}
             {J_{\mathrm{na}}(\varphi)}.
        \label{eq:BJ-threshold}
\end{equation}

We also have an analogue of Corollary \ref{cor:DZ-two-walls} in the polarized setting.
\begin{corollary}
\label{cor:BJ-two-walls}
Let \((X,L)\) be a polarized manifold such that \(\alpha=2\pi c_1(L)\) has positive average scalar curvature \(s>0\). Then
\begin{enumerate}[label=\textup{(\roman*)}]
\item
\(\alpha\) contains a PSC K\"ahler metric if and only if \(\widehat\sigma(X,L)>0\);

\item
\(\alpha\) contains a unique cscK metric if and only if \( \widehat\sigma(X,L)>s\).
\end{enumerate}
\end{corollary}

\begin{proof}
Assertion~\textup{(i)} follows from Theorem \ref{thm:BJ-PSC-criterion}. For assertion~\textup{(ii)}, it follows from \eqref{eq:BJ-NA-F-definition} that
\[
        \widehat\sigma(X,L)-s
        =
        \inf_{\substack{
        \varphi\in\mathcal E^1_{\mathrm{na}}\\
        J_{\mathrm{na}}(\varphi)>0}}
        \frac{M_{\mathrm{na}}(\varphi)}
             {J_{\mathrm{na}}(\varphi)}.
\]
Hence \(\widehat\sigma(X,L)>s\) if and only if there exists \(\delta>0\) such that
\(M_{\mathrm{na}}(\varphi)\geq\delta J_{\mathrm{na}}(\varphi)\) for every \(\varphi\in\mathcal E^1_{\mathrm{na}}\). Since \(J_{\mathrm{na}}\geq0\), with equality precisely on the constant directions, this implies that \((X,L)\) is \(\widehat K\)-polystable. The latter is equivalent to the existence of a cscK metric in \(\alpha\) thanks to \cite[Theorem A]{BoucksomJonsson2025YTD}. Moreover, \cite[Corollary~8.8]{BoucksomJonsson2025YTD} gives \(M_{\mathrm{na}}\equiv0\) on \(\mathcal P_{\mathbb R}\). Then \(J_{\mathrm{na}}=0\) on \(\mathcal P_{\mathbb R}\), and hence that all real product directions are constant. By the Matsushima--Lichnerowicz theorem \cite{Matsushima1957,Lichnerowicz1958}, \(\Aut^0(X,L)\) is reductive.  Since every positive-dimensional connected reductive algebraic group contains a nontrivial algebraic
torus, the absence of nonconstant real product directions forces \(\Aut^0(X,L)\) to be trivial.  The uniqueness of cscK metrics modulo \(\Aut^0(X,L)\) therefore implies that the cscK metric is unique (see \cite[Theorem~1.3]{BermanBerndtsson2017}).

Conversely, suppose that \(\alpha\) contains a unique cscK metric. Then the real product directions are precisely the constant directions. By \cite[Theorem~A]{BoucksomJonsson2025YTD}, \(M_{\mathrm{na}}(\varphi)\geq \delta d_{1,\mathrm{na}}(\varphi,\mathbb R)\) for some \(\delta>0\) and every \(\varphi\in\mathcal E^1_{\mathrm{na}}(L)\). It follows from \eqref{eq:BJ-J-d1-comparison} that \(M_{\mathrm{na}}(\varphi)\geq\delta' J_{\mathrm{na}}(\varphi)\) for some \(\delta'\).  Thus \(\widehat\sigma(X,L)>s\).
\end{proof}

\subsection{Comparison of two thresholds on polarized manifolds}
\label{subsec:threshold-comparison}

The goal of this subsection is to compare the two thresholds \(\sigma\) and \(\widehat \sigma\) through their strict lower levels. By \eqref{eq:BJ-F-radial-identity}, the threshold \(\widehat{\sigma}\) admits the equivalent radial description
\begin{equation}
        \widehat\sigma(X,L)
        =
        \inf_{\substack{
        [\ell]\in\cE^1_{\mathrm{rad}}}\\J\{\ell\}>0}
        \frac{F\{\ell\}}{J\{\ell\}}.
        \label{eq:BJ-threshold-radial}
\end{equation}
Indeed, if \(\varphi=[\ell]\in\cE^1_{\mathrm{na}}\), then
\[
        F\{\ell\}=F_{\mathrm{na}}(\varphi),
        \qquad
        J\{\ell\}=J_{\mathrm{na}}(\varphi),
\]
whereas \(F\{\ell\}=+\infty\) whenever \(\varphi\notin\cE^1_{\mathrm{na}}\).  Thus directions outside \(\cE^1_{\mathrm{na}}\) do not affect the infimum.

\begin{proposition}
\label{prop:strict-level-comparison}
For every \(q\in\R\), the following are equivalent:
\begin{enumerate}[label=\textup{(\roman*)}]
\item \(\sigma(X,\alpha)>q\);

\item \(\widehat\sigma(X,L)>q\);

\item there exist \(\varepsilon>0\) and \(C\in\R\) such that for every \(u\in\cE^1_0(X,\omega)\)
\begin{equation}
        F(u)\geq(q+\varepsilon)J(u)-C.
        \label{eq:strict-level-Archimedean}
\end{equation}

\end{enumerate}
\end{proposition}

\begin{proof}
We prove \(\textup{(i)}\implies \textup{(ii)}\implies\textup{(iii)}\implies\textup{(i)}\).
Assume first that \textup{(i)} holds. By the definition of \(\sigma(X,\alpha)\), there exist \(\beta>0\) and \(\eta>0\) such that for every K\"ahler test configuration \(\mathcal T\) with \(J(\mathcal T)>0\),
\begin{equation}
        F^\beta(\mathcal T)
        \geq
        (q+2\eta)J(\mathcal T)
        \label{eq:fixed-beta-strict-level}
\end{equation}
Let \(\varphi\in\cE^1_{\mathrm{na}}\) satisfy \(J_{\mathrm{na}}(\varphi)>0\), and let \(\ell_\varphi\) be its maximal geodesic ray. Choose \(\varphi_j\in\cH_{\mathrm{na}}\) such that \(d_{1,\mathrm{na}}(\varphi_j,\varphi)\rightarrow0\). Let \(\ell_j\) be the maximal geodesic ray of \(\varphi_j\). By \eqref{eq:BJ-NA-radial-embedding}, \(d_1^c(\ell_j,\ell_\varphi)\rightarrow0\). Each \(\varphi_j\) is represented by an ample test configuration \(\mathcal T_j\) for \((X,L)\), and hence by a K\"ahler test
configuration for \((X,\alpha)\). Since
\[
        J(\mathcal T_j)
        =
        J_{\mathrm{na}}(\varphi_j)
        \longrightarrow
        J_{\mathrm{na}}(\varphi)>0,
\]
equation \eqref{eq:fixed-beta-strict-level} applies to \(\mathcal T_j\) for all sufficiently large \(j\). The \(d_1^c\)-continuity of \(F^\beta\{\cdot\}\) and \(J\{\cdot\}\) therefore gives
\[
        F^\beta\{\ell_\varphi\}
        \geq
        (q+2\eta)J\{\ell_\varphi\}.
\]
Since \(F^\beta\leq F\), while \eqref{eq:BJ-F-radial-identity} gives \(F\{\ell_\varphi\}=F_{\mathrm{na}}(\varphi)\) and \(J\{\ell_\varphi\}=J_{\mathrm{na}}(\varphi)\), we obtain
\[
        F_{\mathrm{na}}(\varphi)
        \geq
        (q+2\eta)J_{\mathrm{na}}(\varphi).
\]
Taking the infimum over all such \(\varphi\) yields
\[
        \widehat\sigma(X,L)
        \geq q+2\eta>q.
\]
Thus \textup{(ii)} holds.

Next, assume \textup{(ii)}. Choose \(\eta>0\) such that
\[
        \widehat\sigma(X,L)>q+2\eta.
\]
By \eqref{eq:BJ-threshold-radial}, for every \([\ell]\in\cE^1_{\mathrm{rad}}\) with \(J\{\ell\}>0\),
\begin{equation}
            F\{\ell\}
        \geq
        (q+2\eta)J\{\ell\}.
        \label{eq:radial-ineq}
\end{equation}
We claim that there exists \(C_1\in\R\) such that
\begin{equation}
        F(u)
        \geq
        (q+\eta)d_1(0,u)-C_1,
        \qquad
        u\in\cE^1_{\sup}(X,\omega).
        \label{eq:strict-level-d1-lower-bound}
\end{equation}
Suppose otherwise. We apply the ray-extraction argument of \cite[Proposition~4.2]{Sha2026PSCVariational}, with \(\sigma=q+\eta\), on the supremum-normalized slice. We briefly explain the normalization adjustment.

If \(u\in\cE^1_{\sup}(X,\omega)\), then \(u\leq0\), and hence \(d_1(0,u)=-E(u)\).
Its \(\cE^1_0\)-representative \(\widetilde u:=u-E(u)\) satisfies
\[
        E(\widetilde u)=0,
        \qquad
        F(\widetilde u)=F(u),
        \qquad
        d_1(0,\widetilde u)
        \leq
        2d_1(0,u).
\]
Consequently, \cite[Proposition~4.1]{Sha2026PSCVariational} remains valid on \(\cE^1_{\sup}(X,\omega)\). Moreover, if \(u_{j,t}\) is the finite-energy geodesic segment joining \(0\) to an endpoint \(u_j\) with \(\sup_Xu_j=0\), then \(\sup_Xu_{j,t}=0\) for every \(t\). Indeed, this follows from the invariance of the upper endpoint slope along a weak geodesic established in \cite[Theorem~3.4]{Darvas2017WeakGeodesics}. Thus the proof of \cite[Proposition~4.2]{Sha2026PSCVariational} produces a unit-speed finite-energy geodesic ray
\(\ell=\{u_t\}_{t\geq0}\subset\cE^1_{\sup}(X,\omega)\) such that
\[
        d_1(0,u_t)=t,
        \qquad
        F\{\ell\}\leq q+\eta.
\]
Let \(C_0\) be the constant in \eqref{eq:sup-slice-d1-J}. Then
\[
        t-C_0
        \leq
        J(u_t)
        \leq
        t.
\]
Dividing by \(t\) and letting \(t\to\infty\), we obtain \(J\{\ell\}=1\). Since \(\ell\) is nonconstant, \eqref{eq:radial-ineq} gives
\[
        F\{\ell\}
        \geq
        (q+2\eta)J\{\ell\}
        =
        q+2\eta,
\]
contradicting \(F\{\ell\}\leq q+\eta\). This proves \eqref{eq:strict-level-d1-lower-bound}.

For \(u\in\cE^1_{\sup}(X,\omega)\), \eqref{eq:sup-slice-d1-J} gives
\[
        (q+\eta)d_1(0,u)
        \geq
        (q+\eta)J(u)-|q+\eta|C_0.
\]
Combining it with \eqref{eq:strict-level-d1-lower-bound} and changing the additive constant gives
\[
        F(u)\geq(q+\eta)J(u)-C.
\]
Since \(F\) and \(J\) are invariant under addition of constants, the same inequality holds on \(\cE^1_0(X,\omega)\). Hence \textup{(iii)} holds, with \(\varepsilon=\eta\).

Finally, assume \textup{(iii)}, and set \(t:=q+\frac{\varepsilon}{2}\). Then \(F-tJ\) is \(d_1\)-coercive on \(\cE^1_0(X,\omega)\). By Proposition \ref{prop:DZ-shifted-quantization}, \(F^\beta-tJ\) is \(d_1\)-coercive for every sufficiently large \(\beta>0\). Fix one such \(\beta\). Equivalently, there exist \(\delta>0\) and \(C_\beta\in\R\) such that for every \(u\in\cE^1_0(X,\omega)\)
\[
        F^\beta(u)-tJ(u)
        \geq
        \delta J(u)-C_\beta,
\]
which is exactly
\[
        F^\beta(\mathcal T)
        \geq
        (t+\delta)J(\mathcal T)
\]
for every K\"ahler test configuration \(\mathcal T\). Consequently,
\[
        \sigma(X,\alpha)
        \geq t+\delta
        >q,
\]
which proves \textup{(i)}.
\end{proof}

We can now prove Theorem \ref{thm:main-BJ}.
\begin{proof}[Proof of Theorem \ref{thm:main-BJ}]
By Proposition \ref{prop:strict-level-comparison}, for every \(q\in\R\), 
\[ 
q<\sigma\bigl(X,2\pi c_1(L)\bigr) \quad\Longleftrightarrow\quad q<\widehat\sigma(X,L). 
\] 
Thus the two extended real numbers have the same strict lower levels. Hence \(\sigma\bigl(X,2\pi c_1(L)\bigr) = \widehat\sigma(X,L)\).
\end{proof}

Consequently, we obtain the following result.
\begin{corollary}
\label{cor:two-walls}
Let \((X,L)\) be a polarized manifold so that \(2\pi c_1(L)\) has positive average scalar curvature \(s>0\). Set \(\sigma:=\sigma\bigl(X,2\pi c_1(L)\bigr)=\widehat\sigma(X,L)\). Then:
\begin{enumerate}[label=\textup{(\roman*)}]
\item
The class \(2\pi c_1(L)\) contains a PSC K\"ahler metric if and only if \(\sigma>0\).

\item
The class \(2\pi c_1(L)\) contains a unique cscK metric if and only if \(\sigma > s\).
\end{enumerate}
\end{corollary}

Thus the common threshold separates the following regimes:
\[
\begin{array}{ccl}
        \sigma\leq0
        &:&
        \text{no PSC K\"ahler metric},\\[1mm]
        0<\sigma<s
        &:&
        \text{PSC K\"ahler metrics, but no cscK metric},\\[1mm]
        \sigma=s
        &:&
        \text{PSC K\"ahler metrics; while cscK existence is not}\\
        &&\text{determined by the threshold alone},\\[1mm]
        \sigma>s
        &:&
        \text{a unique cscK metric}.
\end{array}
\]
Since \(s>0\), the cscK metric in the last regime has positive scalar curvature. At the critical level \(\sigma=s\), the existence of a cscK metric is determined by the full \(\widehat K\)-polystability condition rather than by the numerical threshold.

\bibliographystyle{alpha}
\bibliography{shifted_psc_stability_v2}

\bigskip
  \footnotesize

  Zehao Sha, \textsc{Institute for Mathematics and Fundamental Physics, Hefei, China}\par\nopagebreak
  Email address: \texttt{zhsha@imfp.org.cn}\par\nopagebreak
  Homepage: \url{https://ricciflow19.github.io/}

\end{document}